\documentclass[11pt]{scrartcl}

\usepackage{amsmath,amssymb,amsthm}
\usepackage[english]{babel}
\usepackage[utf8]{inputenc}
\usepackage{csquotes}
\usepackage{tikz}
\usepackage[a4paper,headheight=14pt]{geometry}
\usepackage{enumitem}
\usetikzlibrary{matrix,arrows.meta,quotes,graphs,calc,cd}
\usepackage{csquotes}
\usepackage[backend=biber, style=alphabetic]{biblatex}
\usepackage{hyperref}
\usepackage{microtype}

\makeatletter\def\blx@nowarnpolyglossia{}\makeatother %

\newtheorem{lem}{Lemma}[section]
\newtheorem{prop}[lem]{Proposition}
\newtheorem{thm}[lem]{Theorem}

\newtheorem{corl}[lem]{Corollary}
\theoremstyle{definition}
\newtheorem{dfn}[lem]{Definition}
\theoremstyle{remark}

\newcommand{\C}{{\mathbb{C}}}
\newcommand{\Z}{{\mathbb{Z}}}
\newcommand{\R}{{\mathbb{R}}}
\newcommand{\Q}{{\mathbb{Q}}}
\newcommand{\N}{{\mathbb{N}}}

\newcommand{\ev}{\mathrm{ev}}
\newcommand{\llbracket}{[\kern -0.25em[}
\newcommand{\rrbracket}{]\kern -0.25em]}
\newcommand{\lllbracket}{[\kern -0.25em[\kern -0.25em[}
\newcommand{\rrrbracket}{]\kern -0.25em]\kern -0.25em]}
\newcommand{\bbrack}[1]{\llbracket#1\rrbracket}

\DeclareMathOperator{\id}{id}

\DeclareMathOperator{\as}{as}

\title{Thomsen's D-theory and the K-theoretic Kronecker pairing}
\author{Benedikt Hunger}

\begin{document}

\maketitle

\begin{abstract}
  We calculate Thomsen's D-theory groups $D(\Sigma,B)$, where $\Sigma=C_0(\R)$. Furthermore, we relate the pairing $K_0(A)\times
  E(A,B)\to K_0(B)$ to a similar pairing which is defined using D-theory.
\end{abstract}

\section{Introduction}

Connes and Higson \cite{connes-higson-deformations-morphismes-asymptotiques} introduced the so-called E-theory groups $E(A,B)$
as the morphism sets in an additive category $E$ whose objects are the separable C*-algebras. There is a natural functor from
the category of C*-algebras to the category $E$ which has the universal property that every stable, homotopy-invariant, and
half-exact functor from the category of C*-algebras to an additive category factors through $E$. In particular, E-theory is
closely related to Kasparov's KK-theory \cite{kasparov-the-operator-k-functor}. E-theory was successfully used by Higson and
Kasparov \cite{higson-bivariant-k-theory-and-the-novikov-conjecture,higson-kasparov-operator-k-theory-properly-isometrically}
to prove special cases of the Baum--Connes conjecture.

In \cite{thomsen-discrete-asymptotic-homomorphisms}, Thomsen introduced D-theory, a discrete variant of Connes's and Higson's
E-theory. Like E-theory, D-theory is a bifunctor from the category of separable C*"=algebras to the category of abelian groups,
and there exist products
\begin{align*}
  D(A,B)\times D(B,C)&\to D(A,C),\\
  D(A,B)\times E(B,C)&\to D(A,C),\\
  E(A,B)\times D(B,C)&\to D(A,C),
\end{align*}
relating E-theory and D-theory. 

It is a well-known calculation that the E-theory groups satisfy $E(\C,B)\cong K_0(B)$ for all C*-algebras $B$. However, there
is, up to now, no concrete calculation of D-theory groups at all. Of course, one might hope to calculate $D(\C,B)$ in a similar
way. Since D-theory (like E-theory) supports a suspension isomorphism $\Sigma\colon D(A,B)\overset\cong\to D(\Sigma A,\Sigma B)$
for all separable C*-algebras $A$ and $B$ (where $\Sigma A=C_0(\R)\otimes A$ is the suspension of $A$), it suffices to calculate
the groups $D(\Sigma,B)$ for $\Sigma=\Sigma\C=C_0(\R)$.

Thus, the aim of the first part of this paper will be to prove the existence of a natural isomorphism \begin{equation}
D(\Sigma,B)\cong\frac{\prod_{n\in\N}K_0(B)}{\bigoplus_{n\in\N}K_0(B)}\label{eq:D(S,B)} \end{equation} where $\Sigma=C_0(\R)$ is
the suspension C*-algebra. The proof of \eqref{eq:D(S,B)} is similar, in structure, to the proof of the natural isomorphism
$E(\C,B)\cong K_0(B)$. However, the technical details are slightly more complicated and involve the use of the concrete form of
Cuntz's periodicity map \cite{cuntz-k-theory-c*-algebras}. Product-modulo-sum quotients as in the right hand of
\eqref{eq:D(S,B)} have appeared before in various contexts, for example in the work of Blackadar and Kirchberg
\cite{blackadar-kirchberg-generalized-inductive-limits-of-finite-dimensional-cstar-algebras} on inductive limits of C*-algebras,
of Hanke and Schick \cite{hanke-schick-enlargeability-and-index-theory-I,hanke-schick-enlargeability-and-index-theory-infinite,
hanke-schick-novikov-low-degree-cohomology} or of Carrión and Dadarlat \cite{carrion-dadarlat-almost-flat-k-theory} on the maximal 
Baum--Connes assembly map.

There exists a natural pairing $K_0(A)\times E(A,B)\cong E(\C,A)\times E(A,B)\to E(\C,B)\cong K_0(B)$, defined using the
E-theory product \cite{connes-higson-deformations-morphismes-asymptotiques}. By \eqref{eq:D(S,B)} there is also a pairing
\[
  \frac{\prod_{n\in\N}K_0(A)}{\bigoplus_{n\in\N}K_0(A)}\times E(A,B)\cong D(\Sigma,A)\times E(A,B)\to D(\Sigma,B)
  \cong\frac{\prod_{n\in\N}K_0(B)}{\bigoplus_{n\in\N}K_0(B)}
\]
defined using Thomsen's product. In the second part of this paper, we will show that this pairing agrees with the pairing
$\prod_{n\in\N}K_0(A)/\bigoplus_{n\in\N}K_0(A)\times E(A,B)\to\prod_{n\in\N}K_0(B)/\bigoplus_{n\in\N}K_0(B)$ induced by 
the E-theory product. This shows that it is possible to investigate the asymptotic behaviour of the pairing $K_0(A)\times
E(A,B)\to K_0(B)$ using Thomsen's D-theory product.

This paper is based on part of the author's doctoral dissertation at the Universität Augsburg. The author
would like to thank his advisor Bernhard Hanke for his support and advice, and his helpful remarks on first version of this
paper. The dissertation project was supported by a scholarship of the \emph{Studienstiftung des deutschen Volkes} and by the
TopMath program of the \emph{Elitenetzwerk Bayern}.

\section{E-theory and D-theory}

We will define the D-theory groups, and in particular Thomsen's products, in a way which differs slightly from
\cite{thomsen-discrete-asymptotic-homomorphisms}, and which is inspired by the definition of the corresponding objects for
E-theory in \cite{guentner-higson-trout-equivariant-e-theory}. We will also review the definition of E-theory, as we will need
it later on.

If $X$ is a locally compact Hausdorff space and $B$ is a C*-algebra, we denote by $C_b(X;B)$ the C*-algebra of bounded continuous
$B$-valued continuous functions on $X$, and by $C_0(X;B)\subset C_b(X;B)$ the ideal of functions vanishing at infinity. We are
particularly interested in the special case $X=P=[0,\infty)$, where we abbreviate $\mathcal TB=C_b(P;B)$ and
$\mathcal T_0B=C_0(P;B)$. Similarly, we write $\mathcal T_\delta B=C_b(\N;B)$,
$\mathcal T_{\delta,0}B=C_0(\N;B)$.

\begin{dfn}
  The \emph{discrete asymptotic algebra} \cite{thomsen-discrete-asymptotic-homomorphisms} over $B$ is the C*-algebra $\mathcal
  A_\delta B=\mathcal T_\delta B/\mathcal T_{\delta,0}B$, and the \emph{asymptotic algebra}
  \cite{guentner-higson-trout-equivariant-e-theory} over $B$ is $\mathcal AB=\mathcal TB/\mathcal T_0B$.
\end{dfn}

Every *-ho\-mo\-mor\-phism $f\colon A\to B$ induces, by postcomposition, *-ho\-mo\-mor\-phisms $C_b(X;A)\to C_b(X;B)$ which
restrict to *-ho\-mo\-mor\-phisms $C_0(X;A)\to C_0(X;B)$. In particular, $f$ also induces *-ho\-mo\-mor\-phisms $\mathcal
A_\delta A\to\mathcal A_\delta B$ and $\mathcal AA\to\mathcal AB$. These definitions turn $\mathcal A$ and $\mathcal A_\delta$
into functors, which are easily seen to be exact.

\begin{dfn}
  An \emph{asymptotic homomorphism} between C*-algebras $A$ and $B$ is a *-ho\-mo\-mor\-phism $A\to\mathcal AB$. Similarly, a
  \emph{discrete asymptotic homomorphism} is a *-ho\-mo\-mor\-phism $A\to\mathcal A_\delta B$.
\end{dfn}

The E-theory and D-theory groups are defined as equivalence classes of asymptotic homomorphisms and discrete asymptotic homomorphisms,
respectively, modulo an appropriate type of homotopy. In fact, it turns out that the usual notion of homotopy of
*-ho\-mo\-mor\-phisms is too restrictive for asymptotic homomorphisms, and that the following provides a useful notion of
homotopy between (discrete) asymptotic homomorphisms.

\begin{dfn}
  An \emph{asymptotic homotopy} is a *-ho\-mo\-mor\-phism $H\colon A\to\mathcal AIB$, where $IB=C([0,1],B)$ is the C*-algebra of
  continuous $B$-valued functions on the unit interval $[0,1]$. The asymptotic homomorphisms $\mathcal A\ev_0\circ H$ and
  $\mathcal A\ev_1\circ H$ from $A\to B$ are then called \emph{asymptotically homotopic}. Here $\ev_\tau\colon IB\to B$ denotes
  the evaluation at $\tau\in[0,1]$.

  Similarly, a \emph{discrete asymptotic homotopy} is a *-ho\-mo\-mor\-phism $H\colon A\to\mathcal A_\delta IB$, and again we
  call $\mathcal A_\delta\ev_0\circ H$ and $\mathcal A_\delta\ev_1\circ H$ \emph{asymptotically homotopic}.
\end{dfn}

One can show \cite[Proposition 2.3]{guentner-higson-trout-equivariant-e-theory} that asymptotic homotopy defines an equivalence 
relation on the sets of (discrete) asymptotic homomorphisms from $A$ to $B$. 

\begin{dfn}
  We denote the set of asymptotic homotopy classes of asymptotic homomorphisms from $A$ to $B$ by $\bbrack{A,B}$, and the set of
  asymptotic homotopy classes of discrete asymptotic homomorphisms from $A$ to $B$ by $\bbrack{A,B}_\delta$.
\end{dfn}

\begin{dfn}
  If $A$ and $B$ are separable\footnote{The separability assumption here is in order to make sure that we can define the products 
  later on.} C*-algebras, then we define
  \[
    E(A,B)=\bbrack{\Sigma A\otimes\mathcal K,\Sigma B\otimes\mathcal K}
  \]
  and
  \[
    D(A,B)=\bbrack{\Sigma A\otimes\mathcal K,\Sigma^2B\otimes\mathcal K}_\delta,
  \]
  where $\Sigma A=\{\phi\in IA:\phi(0)=\phi(1)=0\}\cong C_0(\R)\otimes A$ is the \emph{suspension} of $A$, and where 
  $\mathcal K=\mathcal K(\ell^2)$ is the C*-algebra of compact operators on a separable Hilbert space.
\end{dfn}

Note that the definitions of the E-theory and D-theory groups directly imply that there are stability isomorphisms $E(A,B)\cong
E(A\otimes\mathcal K,B)\cong E(A,B\otimes\mathcal K)$ and analogous stability isomorphisms for the D-theory groups.

The double suspension which appears in the definition of $D(A,B)$ in front of the C*-algebra $B$ has its origin in the following
alternative viewpoint towards the discrete asymptotic algebra: The inclusion $\N\subset P=[0,\infty)$ induces restriction maps
$\mathcal TB\to\mathcal T_\delta B$ and $\mathcal T_0B\to\mathcal T_{\delta,0}B$ for all C*-algebras $B$. These maps are clearly
natural, so the induce a natural transformation $\mathcal A\to\mathcal A_\delta$, which is clearly surjective. Thus, we may
define another functor by considering the kernel of this natural transformation.

\begin{dfn}
  The \emph{sequentially trivial asymptotic algebra} \cite{thomsen-discrete-asymptotic-homomorphisms} over $B$ is the C*-algebra
  $\mathcal A_0B=\ker(\mathcal AB\to\mathcal A_\delta B)$. A \emph{sequentially trivial asymptotic homomorphism} is a
  *-ho\-mo\-mor\-phism $A\to\mathcal A_0B$, and a \emph{sequentially trivial asymptotic homotopy} is a 
  *-ho\-mo\-mor\-phism $A\to\mathcal A_0IB$ as before. We denote by $\bbrack{A,B}_0$ the set of asymptotic
  homotopy classes of sequentially trivial asymptotic homomorphisms.
\end{dfn}

There is a natural equivalence between the functors $\mathcal A_\delta\circ\Sigma$ and $\mathcal A_0$
\cite[cf.][Lemma 5.4]{thomsen-discrete-asymptotic-homomorphisms}. Indeed, we may define maps 
$\eta_B\colon\mathcal T_\delta\Sigma B\to\mathcal TB$ by
\[
  \eta_B(\phi)(t)=\phi(\lfloor t\rfloor)(t-\lfloor t\rfloor).
\]
Then $\eta_B(\phi)\in\mathcal T_0B$ whenever $\phi\in\mathcal T_{\delta,0}\Sigma B$, so there is an induced *-ho\-mo\-mor\-phism
$\eta_B\colon\mathcal A_\delta\Sigma B\to\mathcal AB$.

\begin{lem}\label{lem:discrete and sequentially trivial asymptotic homomorphisms}
  The *-ho\-mo\-mor\-phism $\eta_B\colon\mathcal A_\delta\Sigma B\to\mathcal AB$ is injective, and its image equals $\mathcal
  A_0B$.  Thus, $\eta_B\colon\mathcal A_\delta\Sigma B\to\mathcal A_0B$ is a natural *-iso\-mor\-phism.
\end{lem}
\begin{proof}
  It suffices to prove that the rightmost column in the diagram
  \[
    \begin{tikzpicture}
      \matrix (m) [matrix of math nodes, row sep=2em, column sep=2em, commutative diagrams/every cell]
      {
        & 0 & 0 & 0 \\
      0 & \mathcal T_{\delta,0}\Sigma B & \mathcal T_\delta\Sigma B & \mathcal A_\delta\Sigma B & 0 \\
      0 & \mathcal T_0B & \mathcal TB & \mathcal AB & 0 \\
      0 & \mathcal T_{\delta,0}B & \mathcal T_\delta B & \mathcal A_\delta B & 0 \\
        & 0 & 0 & 0 \\
      };
      \path[-stealth, commutative diagrams/.cd, every label] (m-1-2) edge (m-2-2) (m-1-3) edge (m-2-3) (m-1-4) edge (m-2-4)
        (m-2-1) edge (m-2-2) (m-2-2) edge (m-2-3) edge (m-3-2) (m-2-3) edge (m-2-4) edge node [right] {$\eta_B$} (m-3-3) (m-2-4)
        edge (m-2-5) edge node [right] {$\eta_B$} (m-3-4) (m-3-1) edge (m-3-2) (m-3-2) edge (m-3-3) edge (m-4-2) (m-3-3) edge 
        (m-3-4) edge (m-4-3) (m-3-4) edge (m-3-5) edge (m-4-4) (m-4-1) edge (m-4-2) (m-4-2) edge (m-4-3) edge (m-5-2) (m-4-3)
        edge (m-4-4) edge (m-5-3) (m-4-4) edge (m-4-5) edge (m-5-4);
    \end{tikzpicture}
  \]
  is exact. Since the rows of the diagram are exact by definition, it suffices, by the Nine Lemma, to prove that the two left
  columns are exact, which is straightforward.
\end{proof}

Hence, there are natural bijections $\bbrack{A,\Sigma B}_\delta\to\bbrack{A,B}_0$, and we may equally well define $D(A,B)$ by
\[
  D(A,B)=\bbrack{\Sigma A\otimes\mathcal K,\Sigma B\otimes\mathcal K}_0.
\]
This description is most useful for constructing the products relating D-theory and E-theory. These products were introduced by
Connes and Higson \cite{connes-higson-deformations-morphismes-asymptotiques} for E-theory, and by Thomsen
\cite{thomsen-discrete-asymptotic-homomorphisms} for D-theory. The E-theory products were defined in a slightly different way by
Guentner, Higson and Trout \cite{guentner-higson-trout-equivariant-e-theory}. We will give definitions of the products which
combine the methods of Thomsen and of Guentner, Higson and Trout. Indeed, we will use the following statement from
\cite{guentner-higson-trout-equivariant-e-theory} to simplify the definition given in
\cite{thomsen-discrete-asymptotic-homomorphisms} significantly.

\begin{lem}[{\cite[Claim 2.18]{guentner-higson-trout-equivariant-e-theory}}]\label{lem:fundamental-claim-ght}
  Consider a C*-subalgebra $E\subset\mathcal T^2B$ which is separable. Then there exists a piecewise linear invertible continuous 
  function $r_0\colon P\to P$ with $\lim_{t\to\infty}r_0(t)=\infty$, such that
  \[
    \limsup_{t\to\infty}\sup_{r\geq r_0(t)}\|F(t)(r)\|\leq\|[\pi\circ F]\|_{\mathcal A^2B}
  \]
  for all $F\in E$, where $\pi\colon\mathcal TB\to\mathcal AB$ is the projection.\footnote{Hence $\pi\circ F\colon P\to\mathcal
  AB$ represents an element of $\mathcal A^2B$.} Any such function $r_0$ will be called an \emph{admissible reparametrization} for 
  $E\subset\mathcal T^2B$.\qed
\end{lem}

Now the key idea in the definition of the products is the following: Suppose that $f\colon A\to\mathcal AB$ and $g\colon
B\to\mathcal AC$ are asymptotic homomorphisms. Then we may consider $\mathcal Ag\circ f\colon A\to\mathcal A^2C$. Thus, if we
had a natural map $\Phi\colon\mathcal A^2C\to\mathcal AC$ then we could use this map to define the product $g\bullet
f=\Phi\circ\mathcal Ag\circ f\colon A\to\mathcal AC$.

Using Lemma \ref{lem:fundamental-claim-ght}, this idea can be made rigorous provided that $A$ is separable. Indeed, suppose that
$E_0\subset\mathcal A^2C$ is a separable C*-subalgebra. Let $\as_C\colon\mathcal T^2C\to\mathcal A^2C$ be the natural projection.
Then we can choose a separable C*-subalgebra $E\subset\mathcal T^2C$ with $E_0\subset\as_C(E)$, and an admissible
reparametrization $r_0\colon P\to P$ for $E$. We define *-ho\-mo\-mor\-phisms $\Phi,\hat\Phi\colon E_0\to\mathcal AC$ by
\begin{align*}
  \Phi([\pi\circ F])&=[t\mapsto F(t)(r_0(t))],\\
  \hat\Phi([\pi\circ F])&=[t\mapsto F(r_0^{-1}(t))(t)]
\end{align*}
for $F\in E$ with $\as_C(F)=[\pi\circ F]\in E_0$. Lemma \ref{lem:fundamental-claim-ght} makes sure that $\Phi$ and $\hat\Phi$
are well-defined *-ho\-mo\-mor\-phisms which, however, certainly do depend on the choices of $E_0$ and $r_0$.

Now suppose that $f\colon A\to\mathcal AB$ and $g\colon B\to\mathcal AC$ are asymptotic homomorphisms. If $A$ is separable, then
also $\mathcal Ag(f(A))\subset\mathcal A^2C$ is separable. Choose a separable C*-subalgebra $E\subset\mathcal T^2C$ with
$\mathcal Ag(f(A))\subset\as_C(E)$, and fix an admissible reparametrization $r_0\colon P\to P$ for $E$. Let $\Phi\colon\mathcal
Ag(f(A))\to\mathcal AC$ be a as constructed above. Then we define the asymptotic composition of $f$ and $g$ to be $g\bullet
f=\Phi\circ\mathcal Ag\circ f \colon A\to\mathcal AC$. This construction gives a product
$\bbrack{A,B}\times\bbrack{B,C}\to\bbrack{A,C}$, $([f],[g])\mapsto[g\bullet f]$, and hence a product $E(A,B)\times E(B,C)\to
E(A,C)$ if $A$ is separable. It turns out that the choices which went into the definition of $\Phi$ do not change this product,
and that one could equally well replace $\Phi$ by $\hat\Phi$.

Now the situation is very similar when one considers sequentially trivial asymptotic homomorphisms: If $f\colon A\to\mathcal
A_0B$ is sequentially trivial and $g\colon B\to\mathcal AC$ is as above, then we may choose $E\subset\mathcal T^2C$ with
$\mathcal Ag(f(A))\subset\as_C(E)$ in such a way that every $F\in E$ satisfies $F(n)(t)=0$ for all $n\in\N$ and $t\in P$. Then
$\Phi\circ\mathcal Ag\circ f\colon A\to\mathcal A_0C$ is sequentially trivial as well. Similarly, if $f\colon A\to\mathcal AB$
is arbitrary and $g\colon B\to\mathcal A_0C$ is sequentially trivial then we may choose $E$ such that $F\in E$ satisfies
$F(t)(n)=0$ for all $n\in\N$ and $t\in P$. In this case, we may put $g\bullet f=\hat\Phi\circ\mathcal Ag\circ f$. Thus, we have
defined products $\bbrack{A,B}_0\times\bbrack{B,C}\to\bbrack{A,C}_0$ and $\bbrack{A,B}\times\bbrack{B,C}_0\to\bbrack{A,C}_0$.
Using either of the natural maps $\bbrack{A,B}_0\to\bbrack{A,B}$ or $\bbrack{B,C}_0\to\bbrack{B,C}$, we also obtain a product
$\bbrack{A,B}_0\times\bbrack{B,C}_0\to\bbrack{A,C}_0$. All of these products are compatible with the maps
$\bbrack{\cdot,\cdot}_0\to\bbrack{\cdot,\cdot}$, and they satisfy all possible kinds of associativity laws \cite[cf.][Section
3]{thomsen-discrete-asymptotic-homomorphisms}.

There are two natural ways to define a group structure on the sets $E(A,B)$ and $D(A,B)$. The first one, which is employed in
\cite{guentner-higson-trout-equivariant-e-theory} and \cite{thomsen-discrete-asymptotic-homomorphisms}, goes as follows: Choose
a unitary isomorphism $V\colon\ell^2\to\ell^2\oplus\ell^2$. Then conjugation with $V$ induces a C*-algebra isomorphism $\mathcal
K(\ell^2\oplus\ell^2)\to\mathcal K(\ell^2)$, which is independent of the choice of $V$ up to homotopy. We may therefore define a
product on the set $\bbrack{A,\Sigma B\otimes\mathcal K}$ by the composition
\begin{align*}
  \bbrack{A,\Sigma B\otimes\mathcal K}\times\bbrack{A,\Sigma B\otimes\mathcal K}
  &\cong\bbrack{A,\Sigma B\otimes(\mathcal K\oplus\mathcal K)}\\
  &\to \bbrack{A,\Sigma B\otimes\mathcal K(\ell^2\oplus\ell^2)}\cong\bbrack{A,\Sigma B\otimes\mathcal K}.
\end{align*}
Another way would be to employ the natural map $\Sigma\oplus\Sigma\to\Sigma$ given by concatenation, where $\Sigma\cong
C_0(0,1)$. It can be shown that these two products on $\bbrack{A,\Sigma B\otimes\mathcal K}$ agree and define an abelian group
structure on $\bbrack{A,\Sigma B\otimes\mathcal K}$. In an entirely analogous fashion, also $\bbrack{A,\Sigma B\otimes\mathcal
K}_0$ and $\bbrack{A,\Sigma B\otimes\mathcal K}_\delta$ (and hence in particular $D(A,B)$ and $E(A,B)$) are abelian groups. 

One can show that the asymptotic products described above define group homomorphisms $D(A,B)\times D(B,C)\to D(A,C)$,
$D(A,B)\times E(B,C)\to D(A,C)$, $E(A,B)\times D(B,C)\to D(A,C)$, and $E(A,B)\times E(B,C)\to E(A,C)$.

\section{\texorpdfstring{Calculation of $D(\Sigma,B)$}{Calculation of D(Σ,B)}}

In this section, we will give a calculation of the group $D(\Sigma,B)$, for any separable C*-algebra $B$. This calculation bears
some similarities with the well-known natural isomorphism $E(\C,B)\cong K_0(B)$, which we will review first.

Let $B$ be an arbitrary C*-algebra, and consider a unitary $u\in M_n(B_+)$ with $u-1\in M_n(B)$.\footnote{Here $B_+$ is the
unitization of the C*-algebra $B$.} As usual, we denote the set of such unitaries by $U_n^+(B)\subset M_n(B_+)$. In particular,
$u$ represents an element $[u]\in K_1(B)$. We identify the C*-algebra $\Sigma=C_0(0,1)$ with the C*-algebra of all functions
$\phi\in C(S^1)$ with $\phi(1)=0$. Then there exists a unique *-ho\-mo\-mor\-phism $g_u\colon\Sigma\to M_n(B)\subset
B\otimes\mathcal K$ with $g_u(\omega)=u-1$ where $\omega\colon S^1\to\C$ is given by $\omega(z)=z-1$. We define a map
\[
  g^B\colon K_1(B)\to\bbrack{\Sigma,B\otimes\mathcal K}
\]
by $[u]\mapsto[\kappa_{B\otimes\mathcal K}\circ g_u]$ where $\kappa_{B\otimes\mathcal K}\colon B\otimes\mathcal K\to\mathcal
A(B\otimes\mathcal K)$ is given by $\kappa_{B\otimes\mathcal K}(x)=[t\mapsto x]$. The following statement is well-known:

\begin{prop}[{\cite[Theorem 4.1]{rosenberg-role-of-ktheory} and 
  \cite[Proposition 2.19]{guentner-higson-trout-equivariant-e-theory}}]\label{prop:K1 and asymptotic homotopy classes}
  For every C*-algebra $B$, the map $g^{\Sigma B}\colon K_1(\Sigma B)\to\bbrack{\Sigma,\Sigma B\otimes\mathcal K}$ is an 
  isomorphism of groups.\qed
\end{prop}

One can show easily that the inclusion $\Sigma\to\Sigma\otimes\mathcal K$ induces an isomorphism
$E(\C,B)=\bbrack{\Sigma\otimes\mathcal K,\Sigma B\otimes\mathcal K}\to\bbrack{\Sigma,\Sigma B\otimes\mathcal K}$. Together with
Proposition \ref{prop:K1 and asymptotic homotopy classes}, this establishes the isomorphism $K_0(B)\cong K_1(\Sigma B)\cong
E(\C,B)$ for arbitrary C*-algebras $B$.

Now let us turn to the calculation of $D(\Sigma,B)$, which has not appeared in the literature so far. Let again $B$ be a
C*-algebra, and let $(u_n)_{n\in\N}$ be a sequence in $\bigcup_{k\in\N}U_k^+(B)$, so that each $u_n$ represents an element of
$K_1(B)$. The map $\phi\colon\N\to B\otimes\mathcal K$, which is defined by $\phi(n)=u_n-1$, determines an element
$[\phi]\in\mathcal A_\delta(B\otimes\mathcal K)$ such that $[\phi]+1\in\mathcal A_\delta(B\otimes\mathcal K)_+$ is unitary.
Hence there exists a unique discrete asymptotic homomorphism $\tilde g_{(u_n)}\colon\Sigma\to\mathcal A_\delta(B\otimes\mathcal
K)$ such that $\tilde g_{(u_n)}(\omega)=[\phi]=[n\mapsto u_n-1]$. The map
\begin{align*}
  g_\delta^B\colon\prod_{n\in\N}K_1(B)&\to\bbrack{\Sigma,B\otimes\mathcal K}_\delta,\\
  ([u_n])_{n\in\N}&\mapsto[\tilde g_{(u_n)}],
\end{align*}
is well-defined: Indeed if we are given continuous paths $(u_n^\tau)_{\tau\in[0,1]}$ in $U_{k(n)}^+(B)$ then the same
construction as above yields a discrete asymptotic homotopy $H\colon\Sigma\to\mathcal A_\delta I(B\otimes\mathcal K)$
with $H(\omega)=[n\mapsto(\tau\mapsto u_n^\tau-1)]$, and $H$ is a discrete asymptotic homotopy connecting $\tilde g_{(u_n^0)}$
and $\tilde g_{(u_n^1)}$. The key step in calculating $D(\Sigma,B)$ is the following analogue of Proposition \ref{prop:K1 and
asymptotic homotopy classes}.

\begin{prop}\label{prop:g-delta}
  For every C*-algebra $B$, the map
  $g_\delta^{\Sigma B}\colon\prod_{n\in\N}K_1(\Sigma B)\to\bbrack{\Sigma,\Sigma B\otimes\mathcal K}_\delta$ is a surjective group
  homomorphism with
  \[
    \ker g_\delta^{\Sigma B}=\bigoplus_{n\in\N}K_1(\Sigma B),
  \]
  where $\bigoplus_{n\in\N}K_1(\Sigma B)\subset\prod_{n\in\N}K_1(\Sigma B)$ is the subgroup consisting of all sequences 
  $([u_n])_{n\in\N}$ which vanish eventually.
\end{prop}
\begin{proof}
  We begin by proving that $g_\delta^{\Sigma B}$ is surjective. Thus, we consider an arbitrary element
  $[h]\in\bbrack{\Sigma,\Sigma B\otimes\mathcal K}_\delta$ which is represented by a discrete asymptotic homomorphism
  $h\colon\Sigma\to\mathcal A_\delta(\Sigma B\otimes\mathcal K)$. We may write $h(\omega)=[G]$ for a map $G\colon\N\to
  \Sigma B\otimes\mathcal K$. We may replace each $G(n)$ by an element of
  $\bigcup_{k\in\N}M_k(\Sigma B)\subset\Sigma B\otimes\mathcal K$ which is $n^{-1}$-close to $G(n)$, without altering
  $[G]\in\mathcal A_\delta(\Sigma B\otimes\mathcal K)$. Thus, we may assume that $G(n)\in\bigcup_{k\in\N}M_k(\Sigma B)$ for
  all $n\in\N$.

  We will show next that there exists a map $U\colon\N\to\Sigma B\otimes\mathcal K$ such that
  $U(n)\in\bigcup_{k\in\N}U_k^+(\Sigma B)$ for all $n\in\N$, and such that $[G]=[U-1]\in\mathcal
  A_\delta(\Sigma B\otimes\mathcal K)$. Indeed, $[G+1]\in\mathcal A_\delta(\Sigma B\otimes\mathcal K)$ must be unitary, so
  that 
  \[
    \lim_{n\to\infty}(G(n)+1)^*(G(n)+1)=1.
  \]
  We put $F(n)=G(n)+1$. Thus, $F(n)^*F(n)$ is invertible if $n$ is sufficiently large. Without loss of generality, $F(n)^*F(n)$
  is invertible for all $n\in\N$. Now we put $U(n)=F(n)(F(n)^*F(n))^{-1/2}$. It is straightforward to see that indeed
  $[G]=[F-1]=[U-1]\in\mathcal A_\delta(\Sigma B\otimes\mathcal K)$ and that each $U(n)$ is contained in
  $\bigcup_{k\in\N}U_k^+(\Sigma B)$. We have $g_\delta^{\Sigma B}([U(n)])_{n\in\N}=[\tilde g_{(U(n))}]$ where $\tilde
  g_{(U(n))}\colon\Sigma\to\mathcal A_\delta(\Sigma B\otimes\mathcal K)$ is determined by the property
  \[
    \tilde g_{(U(n))}(\omega)=[n\mapsto U(n)-1]=[U-1]=[G]=h(\omega).
  \]
  Hence, $g_\delta^{\Sigma B}([U(n)])_{n\in\N}=[\tilde g_{(U(n))}]=[h]$ and $g_\delta^{\Sigma B}$ is surjective.

  Next suppose that $(u_n)_{n\in\N}$ is a sequence of unitaries in $\bigcup_{k\in\N}U_k^+(\Sigma B)$ such that $g_\delta^{\Sigma
  B}([U(n)])_{n\in\N}=0\in\bbrack{\Sigma,\Sigma B\otimes\mathcal K}_\delta$. Thus, there exists a discrete asymptotic homotopy
  $H\colon\Sigma\to\mathcal A_\delta I(\Sigma B\otimes\mathcal K)$ with $\mathcal A\ev_0\circ H=\tilde g_{(u_n)}$ and $\mathcal
  A\ev_1\circ H=0$. As above, we may write $H(\omega)=[U-1]$ where $U\colon\N\to(I(\Sigma B\otimes\mathcal K))_+$ is a
  unitary-valued map with $U(n)-1\in I(\Sigma B\otimes\mathcal K)$ for all $n\in\N$. By assumption,
  $\lim_{n\to\infty}\|U(n)(0)-u_n\|=\lim_{n\to\infty}\|U(n)(1)-1\|=0$. A standard argument shows that therefore
  $[u_n]=[U(n)(0)]=[U(n)(1)]=[1]=0\in K_1(\Sigma B)$ if $n$ is sufficiently large. Thus,
  $([u_n])_{n\in\N}\in\bigoplus_{n\in\N}K_1(\Sigma B)$.

  On the other hand, if $([u_n])_{n\in\N}\in\bigoplus_{n\in\N}K_1(\Sigma B)$ then $\tilde g_{(u_n)}(\omega)=[n\mapsto
  u_n-1]=[n\mapsto 0]=0$, so that $g_\delta^{\Sigma B}([u_n])_{n\in\N}=0$. This completes the calculation of $\ker
  g_\delta^{\Sigma B}$.

  Finally, it remains to prove that $g_\delta^{\Sigma B}$ is additive. Thus, let $(u_n)_{n\in\N}$ and $(v_n)_{n\in\N}$ be two
  sequences in $\bigcup_{k\in\N}U_k^+(\Sigma B)$. It is a well-known fact that $[u_n]+[v_n]=[u_nv_n]=[u_n*v_n]\in K_1(\Sigma
  B)$, where $u_n*v_n$ is the concatenation of $u_n$ and $v_n$, viewed as elements of $\Sigma(M_k(B)_+)$. In particular,
  $g_\delta^{\Sigma B}([u_n]+[v_n])_{n\in\N}=[\tilde g_{(u_n*v_n)}]$, and by definition of the group structure on
  $\bbrack{\Sigma,\Sigma B\otimes\mathcal K}_\delta$ we obtain that indeed $[\tilde g_{(u_n*v_n)}]=[\tilde g_{(u_n)}]+[\tilde
  g_{(v_n)}] =g_\delta^{\Sigma B}([u_n])_{n\in\N}+g_\delta^{\Sigma B}([v_n])_{n\in\N}$.
\end{proof}

Note that Proposition \ref{prop:g-delta}, together with Bott periodicity for K-theory and for D-theory and the stability
isomorphism for D-theory, yields a chain of isomorphisms
\begin{align*}
  \frac{\prod_{n\in\N}K_0(B)}{\bigoplus_{n\in\N}K_0(B)}&\cong\frac{\prod_{n\in\N}K_1(\Sigma^3B)}{\bigoplus_{n\in\N}K_1(\Sigma^3B)}
  \cong\bbrack{\Sigma,\Sigma^3B\otimes\mathcal K}_\delta\\
  &\cong\bbrack{\Sigma^2\otimes\mathcal K,\Sigma^4B\otimes\mathcal K\otimes\mathcal K}_\delta\\
  &=D(\Sigma,\Sigma^2B\otimes\mathcal K)\cong D(\Sigma,B).
\end{align*}
However, we will give a more succinct description of this isomorphism next.
Suppose for the moment that $B$ is a unital C*-algebra. Then we define
\[
  \Psi_B\colon\prod_{n\in\N}K_0(B)\to D(\Sigma,B)
\]
as follows: We write an element of $\prod_{n\in\N}K_0(B)$ as $([p_n])_{n\in\N}$ where each $p_n$ is a projection in
$M_\infty(B)=\bigcup_{k\in\N}M_k(B)$. We consider the discrete asymptotic homomorphism 
$f_{(p_n)}\colon\C\to\mathcal A_\delta(B\otimes\mathcal K)$ which is determined by $f_{(p_n)}(1)=[n\mapsto p_n]$. 
Now we define $\Psi_B$ by the prescription $\Psi_B([p_n])_{n\in\N}=[\Sigma^2f_{(p_n)}\otimes\id_{\mathcal
K}]\in\bbrack{\Sigma^2\otimes\mathcal K,\Sigma^2B\otimes\mathcal K\otimes\mathcal K}_\delta=D(\Sigma,B\otimes\mathcal
K)\cong D(\Sigma,B)$. Of course,
\[
  \Sigma^2f_{(p_n)}\otimes\id_{\mathcal K}(\phi\otimes\psi\otimes T)=[n\mapsto\phi\otimes\psi\otimes p_n\otimes T]
\]
for all $\phi,\psi\in\Sigma$ and $T\in\mathcal K$. If $B$ is non-unital we define $\Psi_B$ by requiring that the diagram
\[
  \begin{tikzpicture}
    \matrix (m) [matrix of math nodes, row sep=2em, column sep=2em, commutative diagrams/every cell]
    {
      0 & \prod_{n\in\N}K_0(B) & \prod_{n\in\N}K_0(B_+) & \prod_{n\in\N}K_0(\C) & 0 \\
      0 & D(\Sigma,B) & D(\Sigma,B_+) & D(\Sigma,\C) & 0 \\
    };
    \path[-stealth, commutative diagrams/.cd, every label] (m-1-1) edge (m-1-2) (m-1-2) edge (m-1-3) edge [dashed] node [left]
      {$\Psi_B$} (m-2-2) (m-1-3) edge (m-1-4) edge node [left] {$\Psi_{B_+}$} (m-2-3) (m-1-4) edge (m-1-5) edge node [left]
      {$\Psi_\C$} (m-2-4) (m-2-1) edge (m-2-2) (m-2-2) edge (m-2-3) (m-2-3) edge (m-2-4) (m-2-4) edge (m-2-5);
  \end{tikzpicture}
\]
with exact rows commutes.

\begin{thm}\label{thm:calculation of D(Sigma,B)}
  For any C*-algebra $B$, the map $\Psi_B\colon\prod_{n\in\N}K_0(B)\to D(\Sigma,B)$ is a natural surjective group homomorphism
  with $\ker\Psi_B=\bigoplus_{n\in\N}K_0(B)$.
\end{thm}
\begin{proof}
  It is clear that $\Psi_B$ is natural in $B$. The proof consists of several parts: First we will assume that $B$ is unital in
  order to give another description of the map $\Psi_B$. By naturality, we can extend this description to non-unital
  C*-algebras. Secondly we will use this alternative description for the C*-algebra $\Sigma^2B$ in order to prove the statement
  of the theorem for double suspensions $\Sigma^2B$. Finally we will use the concrete description of Cuntz's version of Bott
  periodicity \cite{cuntz-k-theory-c*-algebras} to reduce the general case to the case of double suspensions.

  We define a map $\bar\Psi_B\colon\prod_{n\in\N}K_0(B)\to D(\Sigma,B\otimes\mathcal K)$ to be the composition
  \begin{align*}
    \prod_{n\in\N}K_0(B)&\xrightarrow{\beta}\prod_{n\in\N}K_1(\Sigma B)
    \xrightarrow{g_\delta^{\Sigma B}}\bbrack{\Sigma,\Sigma B\otimes\mathcal K}\xrightarrow{\Sigma{-}\otimes\id_{\mathcal K}}\\
                        &\to\bbrack{\Sigma^2\otimes\mathcal K,\Sigma^2B\otimes\mathcal K\otimes\mathcal K}_\delta
                        =D(\Sigma,B\otimes\mathcal K)\cong D(\Sigma,B),
  \end{align*}
  where $\beta$ is the Bott periodicity isomorphism.  It is clear that $\bar\Psi_B$ is natural in $B$. We will prove that
  $\Psi_B=\bar\Psi_B$. By naturality, it suffices to prove this for unital C*-algebras $B$.

  Thus, let $B$ be a unital C*-algebra, and let $p_n\in\bigcup_{k\in\N}M_k(B)$ be a sequence of projections, representing an
  element $([p_n])_{n\in\N}\in\prod_{n\in\N}K_0(B)$. By the description of the Bott periodicity isomorphism $K_0(B)\to
  K_1(\Sigma B)$, we have $\beta([p_n])_{n\in\N}=([\omega\otimes p_n+1])_{n\in\N}$. In particular, $g_\delta^{\Sigma
  B}\circ\beta([p_n])_{n\in\N}=[\tilde g]$ where $\tilde g\colon\Sigma\to\mathcal A_\delta(\Sigma B\otimes\mathcal K)$ is such
  that $\tilde g(\omega)=[n\mapsto\omega\otimes p_n]$. Of course, this implies that $\tilde g(\psi)=[n\mapsto\psi\otimes p_n]$
  for all $\psi\in\Sigma$. In particular, $\Sigma\tilde g\otimes\id_{\mathcal K}\colon\Sigma^2\otimes\mathcal K\to\mathcal
  A_\delta(\Sigma^2B\otimes\mathcal K\otimes\mathcal K)$ is such that $\Sigma\tilde g\otimes\id_{\mathcal
  K}(\phi\otimes\psi\otimes T)=[n\mapsto\phi\otimes\psi\otimes p_n\otimes T]$ for all $\phi,\psi\in\Sigma$ and $T\in\mathcal K$.
  Therefore, $\Sigma\tilde g\otimes\id_{\mathcal K}=\Sigma^2f_{(p_n)}\otimes\id_{\mathcal K}$, which implies that indeed
  $\Psi_B=\bar\Psi_B$.

  Now we consider a double suspension $\Sigma^2B$. In this case, $\bar\Psi_B$ can also be written as the composition
  \begin{align*}
    \prod_{n\in\N}K_0(\Sigma^2B)&\xrightarrow{\beta}\prod_{n\in\N}K_1(\Sigma^3B)
    \xrightarrow{g_\delta^{\Sigma^3B}}\bbrack{\Sigma,\Sigma^3B\otimes\mathcal K}_\delta
    \cong\bbrack{\Sigma,\Sigma^2B\otimes\mathcal K}_0\\
                                &\xrightarrow{{-}\otimes\id_{\mathcal K}}\bbrack{\Sigma\otimes\mathcal
                                K,\Sigma^2B\otimes\mathcal K\otimes\mathcal K}_0=D(\C,\Sigma B\otimes\mathcal K)\\
                                &\xrightarrow{\Sigma}D(\Sigma,\Sigma^2B\otimes\mathcal K)\cong D(\Sigma,\Sigma^2B),
  \end{align*}
  where all maps in this composition are isorphisms, except for $g_\delta^{\Sigma^3B}$ which is surjective with kernel equal to
  $\bigoplus_{n\in\N}K_1(\Sigma^3B)$ by Proposition \ref{prop:g-delta}. Since
  $\beta^{-1}(\bigoplus_{n\in\N}K_1(\Sigma^3B))=\bigoplus_{n\in\N}K_0(\Sigma^2B)$, this implies the claim of the theorem for
  $\Sigma^2B$.

  In the case of general $B$, we consider the diagram
  \[
    \begin{tikzpicture}
      \matrix (m) [matrix of math nodes, row sep=2em, column sep=2em, commutative diagrams/every cell]
      {
        0 & \bigoplus_{n\in\N}K_0(\Sigma^2B) & \prod_{n\in\N}K_0(\Sigma^2B) &[1em] D(\Sigma,\Sigma^2B) & 0 \\
        0 & \bigoplus_{n\in\N}K_0(B) & \prod_{n\in\N}K_0(B) & D(\Sigma,B) & 0 \\
      };
      \path[-stealth, commutative diagrams/.cd, every label] (m-1-1) edge (m-1-2) (m-1-2) edge (m-1-3) edge node [left]
        {$\cong$} (m-2-2) (m-1-3) edge node [above] {$\Psi_{S^2B}$} (m-1-4) edge node [left] {$\cong$} (m-2-3) (m-1-4) edge
        (m-1-5) edge node [left] {$\cong$} (m-2-4) (m-2-1) edge (m-2-2) (m-2-2) edge (m-2-3) (m-2-3) edge node [above]
        {$\Psi_B$} (m-2-4) (m-2-4) edge (m-2-5);
    \end{tikzpicture}
  \]
  where the vertical arrows are the periodicity isomorphisms coming from Cuntz's \cite{cuntz-k-theory-c*-algebras} version of
  Bott periodicity. Recall that these periodicity isomorphisms are the index maps associated to a certain short exact
  sequence of C*-algebras. By construction of these index maps, the diagram above commutes because the horizontal maps are given
  by natural transformations. We have already seen that the top row in the diagram is exact, so the bottom row must be exact as
  well.
\end{proof}

\section{The K-theoretical Kronecker pairing and D-theory}

As mentioned in the introduction, the calculation 
\[
  D(\Sigma,A)\cong\frac{\prod_{n\in\N}K_0(A)}{\bigoplus_{n\in\N}K_0(A)}
\]
implies that there are two different ways for defining the product of an element of
$\prod_{n\in\N}K_0(A)/\bigoplus_{n\in\N}K_0(A)$ with an E-theory class in $E(A,B)$, yielding an element of
$\prod_{n\in\N}K_0(B)/\bigoplus_{n\in\N}K_0(B)$. The following result states that these two products are in fact equal.

\begin{thm}\label{thm:kronecker pairing and d-theory}
  Let $A$ and $B$ be C*-algebras and fix $\eta\in E(A,B)$. Then the compositions
  \[
    \frac{\prod_{n\in\N}K_0(A)}{\bigoplus_{n\in\N}K_0(A)}
    \cong\frac{\prod_{n\in\N}E(\C,A)}{\bigoplus_{n\in\N}E(\C,A)}
    \to\frac{\prod_{n\in\N}E(\C,B)}{\bigoplus_{n\in\N}E(\C,B)}
    \cong\frac{\prod_{n\in\N}K_0(B)}{\bigoplus_{n\in\N}K_0(B)}
  \]
  and
  \[
    \frac{\prod_{n\in\N}K_0(A)}{\bigoplus_{n\in\N}K_0(A)}
    \cong D(S\C,A)\to D(S\C,B)
    \cong\frac{\prod_{n\in\N}K_0(B)}{\bigoplus_{n\in\N}K_0(B)},
  \]
  which are given by the respective composition product with $\eta$, coincide.
\end{thm}

Before we step into the proof of the theorem, we state a lemma that we will need in the course of the proof.

\begin{lem}\label{lem:evaluation and Phi}
  Let $B$ be a C*-algebra, and fix $\tau\in[0,1]$. Assume that $\tilde E\subset\mathcal A^2IB$ is separable, and put $\tilde
  E_\tau=\mathcal A^2\ev_\tau(\tilde E)\subset\mathcal A^2B$. Let $E\subset\mathcal T^2IB$ be separable with $\as_{IB}(E)=\tilde
  E$, and put $E_\tau=\mathcal T^2\ev_\tau(E)\subset\mathcal T^2B$ (so that in particular $\as_B(E_\tau)=\tilde E_\tau$). Let
  $r_0\colon P\to P$ be a reparametrization which is admissible for both $E$ and $E_\tau$. Then the diagram
  \[
    \begin{tikzpicture}
      \matrix (m) [matrix of math nodes, row sep=3em, column sep=3em, commutative diagrams/every cell]
      {
        \tilde E & \mathcal AIB \\
        \tilde E_\tau & \mathcal AB \\
      };
      \path[-stealth, commutative diagrams/.cd, every label] (m-1-1) edge node [above] {$\Phi$} (m-1-2) edge node [left]
        {$\mathcal A^2\ev_\tau$} (m-2-1) (m-1-2) edge node [right] {$\mathcal A\ev_\tau$} (m-2-2) (m-2-1) edge node [above]
        {$\Phi$} (m-2-2);
    \end{tikzpicture}
  \]
  commutes if we use $r_0$ to define both horizontal maps $\Phi$.
\end{lem}
\begin{proof}
  Let $F\in E$ be arbitrary. Then $\Phi[\pi\circ F]=[t\mapsto F(t)(r_0(t))]\in\mathcal AIB$ and hence
  \[
    \mathcal A\ev_\tau\circ\Phi[\pi\circ F]=[t\mapsto F(t)(r_0(t))(\tau)]\in\mathcal AB.
  \]
  On the other hand, $\mathcal A^2\ev_\tau[\pi\circ F]=[\pi\circ F_\tau]$ where $F_\tau=\mathcal T^2\ev_\tau(F)\in E_\tau$.
  Therefore,
  \begin{align*}
    \Phi\circ\mathcal A^2\ev_\tau[\pi\circ F]&=\Phi[\pi\circ F_\tau]=[t\mapsto F_\tau(t)(r_0(t))]\\
                                             &=[t\mapsto F(t)(r_0(t))(\tau)]=\mathcal A\ev_\tau\circ\Phi[\pi\circ F]
  \end{align*}
  as claimed.
\end{proof}

\begin{proof}[Proof of Theorem \ref{thm:kronecker pairing and d-theory}]
  By a naturality argument, we may assume without loss of generality that $A$ and $B$ are unital. Let $(p_n)_{n\in\N}$ be a 
  sequence of projections in $A\otimes\mathcal K$, and represent $\eta$ by an asympttic homomorphism $f\colon A\to\mathcal AB$.
  Then the image of $[([p_n])_{n\in\N}]\in\prod_{n\in\N}K_0(A)/\bigoplus_{n\in\N}K_0(A)$ under the first composition is
  represented by a family $(\tilde p_n)_{n\in\N}$ of projections in $B\otimes\mathcal K$ which have the property that the
  asymptotic homotopy classes of the asymptotic homomorphisms $\phi\mapsto[t\mapsto\phi\otimes\tilde p_n]$ and $\phi\mapsto
  f(\phi\otimes p_n)$ agree. Let $H_n\colon\Sigma\to\mathcal AI(\Sigma B\otimes\mathcal K)$ be asymptotic homotopies with
  $\mathcal A\ev_0\circ H_n(\phi)=f(\phi\otimes p_n)$ and $\mathcal A\ev_1\circ H_n(\phi)=[t\mapsto\phi\otimes\tilde p_n]$ for
  all $\phi\in\Sigma$.

  We have to prove that the second composition in the statement of the theorem maps $[([p_n])_{n\in\N}]$ to $[([\tilde
  p_n])_{n\in\N}]$ as well. It follows from the description of $\Psi_A$ in Theorem \ref{thm:calculation of D(Sigma,B)} that
  under the identification
  \[
    \frac{\prod_{n\in\N}K_0(A)}{\bigoplus_{n\in\N}K_0(A)}\cong\bbrack{\Sigma^2,\Sigma^2A\otimes\mathcal K}_\delta
    \cong\bbrack{\Sigma^2,\Sigma A\otimes\mathcal K}_0,
  \]
  $[([p_n])_{n\in\N}]$ is first mapped to the class of the discrete asymptotic homomorphism
  $\psi\otimes\phi\mapsto[n\mapsto\psi\otimes\phi\otimes p_n]$, and then to the class of the sequentially trivial asymptotic
  homomorphism $g\colon\Sigma^2\to\mathcal A_0(\Sigma A\otimes\mathcal K)$ which is given by
  \[
    g(\psi\otimes\phi)=[t\mapsto\psi(t-\lfloor t\rfloor)\phi\otimes p_{\lfloor t\rfloor}].
  \]
  Analogously, $[([\tilde p_n])_{n\in\N}]$ is identified with the class of the sequentially trivial homomorphism $\tilde
  g\colon\Sigma^2\to\mathcal A_0(\Sigma B\otimes\mathcal K)$ which is given by $\tilde
  g(\psi\otimes\phi)=[t\mapsto\psi(t-\lfloor t\rfloor)\phi\otimes\tilde p_{\lfloor t\rfloor}]$. Thus, we have to prove that the
  sequentially trivial asymptotic homomorphisms $f\bullet g$ and $\tilde g$ are asymptotically homotopic. For appropriate
  choices in the respective definitions of $\Phi$ we have
  \begin{align*}
    f\bullet g(\psi\otimes\phi)&=\Phi\left(\mathcal Af[t\mapsto\psi(t-\lfloor t\rfloor)\phi\otimes p_{\lfloor t\rfloor}]\right)\\
                               &=\Phi\left[t\mapsto\psi(t-\lfloor t\rfloor)f(\phi\otimes p_{\lfloor t\rfloor})\right]\\
                               &=\Phi\left[t\mapsto\psi(t-\lfloor t\rfloor)\mathcal A\ev_0\circ H_{\lfloor
                               t\rfloor}(\phi)\right]\\
                               &=\Phi\circ\mathcal A^2\ev_0\left[t\mapsto\psi(t-\lfloor t\rfloor)H_{\lfloor
                               t\rfloor}(\phi)\right]\\
                               &=\mathcal A\ev_0\circ\Phi\left[t\mapsto\psi(t-\lfloor t\rfloor)H_{\lfloor
                               t\rfloor}(\phi)\right],
  \end{align*}
  where the last equality is due to Lemma \ref{lem:evaluation and Phi}. Of course, this is asymptotically homotopic to the
  discrete asymptotic homomorphism $\psi\otimes\phi\mapsto\mathcal A\ev_1\circ\Phi[t\mapsto\psi(t-\lfloor t\rfloor)H_{\lfloor
  t\rfloor}(\phi)]$, and again by Lemma \ref{lem:evaluation and Phi} we have
  \begin{align*}
    \mathcal A\ev_1\circ\Phi\left[t\mapsto\psi(t-\lfloor t\rfloor)H_{\lfloor t\rfloor}(\phi)\right]
    &=\Phi\left[t\mapsto\psi(t-\lfloor t\rfloor)\mathcal A\ev_1\circ H_{\lfloor t\rfloor}(\phi)\right]\\
    &=\Phi\left[t\mapsto\psi(t-\lfloor t\rfloor)\left[s\mapsto\phi\otimes\tilde p_{\lfloor t\rfloor}\right]\right]\\
    &=\left[t\mapsto\psi(t-\lfloor t\rfloor)\phi\otimes\tilde p_{\lfloor t\rfloor}\right]=\tilde g(\psi\otimes\phi).
  \end{align*}
  Thus, indeed $f\bullet g$ and $\tilde g$ are asymptotically homotopic.
\end{proof}

There is a notable special case of Theorem \ref{thm:kronecker pairing and d-theory}. In order to state it, we recall that for
any $j\in\N$, the $j$-th \emph{K-homology group} of a C*-algebra $B$ is given by $K^j(B)=E(B,\Sigma^j\C)$. In particular, if
$\eta\in K^j(B)$ is a K-homology class and $\xi\in K_\ell(B\otimes A)\cong E(\C,\Sigma^\ell(B\otimes A))$ is a K-theory class,
we can define the \emph{Kronecker pairing} of $\eta$ and $\xi$ to be
\[
  \langle\eta,\xi\rangle=(\Sigma^\ell\eta\otimes\id_A)\bullet\xi\in E(\C,\Sigma^{\ell+j}A)\cong K_{\ell+j}(A).
\]
Now suppose that $(\xi_n)_{n\in\N}$ is a sequence in $K_\ell(B\otimes A)$. This sequence then defines a
class in $\prod_{n\in\N}K_\ell(B\otimes A)/\bigoplus_{n\in\N}K_\ell(B\otimes A)\cong D(\Sigma,\Sigma^\ell(B\otimes A))$, and the 
sequence of Kronecker pairings $(\langle\eta,\xi_n\rangle)_{n\in\N}$ likewise defines an element of
$D(\Sigma,\Sigma^{\ell+j}A)$. Now Theorem \ref{thm:kronecker pairing and d-theory} directly implies the following:

\begin{corl}
  In this situation,
  \[
    \Psi_{\Sigma^{\ell+j}A}\left[(\langle\eta,\xi_n\rangle)_{n\in\N}\right]=(\Sigma^\ell\eta\otimes\id_A)\bullet
    \Psi_{\Sigma^\ell(B\otimes A)}\left[(\xi_n)_{n\in\N}\right]\in D(\Sigma,\Sigma^{\ell+j}A),
  \]
  where the right-hand side is defined using the composition product $E(\Sigma^\ell(B\otimes A),\Sigma^{\ell+j}A)\times
  D(\Sigma,\Sigma^\ell(B\otimes A))\to D(\Sigma,\Sigma^{\ell+j}A)$.\qed
\end{corl}

Thus, it is possible to use the D-theory product to calculate Kronecker pairings asymptotically.

\printbibliography

\end{document}